\documentclass[12pt]{amsart} %amsfonts,
\usepackage{amssymb, amsmath, amsthm, amsfonts}
\usepackage[italian, english]{babel}
\usepackage{graphicx}
\usepackage[T1]{fontenc}
\usepackage[latin1]{inputenc}
\usepackage{times}

\newtheorem{thm}{Theorem}[section]
\newtheorem{lemma}[thm]{Lemma}
\newtheorem{defn}[thm]{Definition}
\newtheorem{corollary}[thm]{Corollary}
\newtheorem{proposition}[thm]{Proposition}
\newtheorem{remark}[thm]{Remark}

    \newcommand{\ct}[1]{\langle {#1}\rangle \lower.3ex\hbox{$_{t}$}}
    \newcommand{\lt}[1]{[ {#1}] \lower.3ex\hbox{$_{t}$}}

\newcommand{\R}{{\mathbb{R}}}
\newcommand{\N}{{\mathbb N}}

\newcommand{\C}{{\mathcal C}}

\newcommand{\K}{{\mathcal K}}
\newcommand{\sfe}{{\mathbb S}^{n-1}}

\newcommand{\cl}{{\rm cl}}
\newcommand{\interno}{{\rm int}}

\newcommand{\supp}{{\rm supp}}

\begin{document}
\selectlanguage{english}
\title{Valuations on the space of quasi-concave functions}
\author{Andrea Colesanti, Nico Lombardi}
\address{Dipartimento di Matematica e Informatica ``U. Dini",
Universit\`a degli Studi di Firenze}
\email{colesant@math.unifi.it, nico.lombardi@stud.unifi.it}
%\thanks{Here come thanks!}
\maketitle

\begin{abstract} We characterize the valuations on the space of quasi-concave functions on $\R^N$, that are rigid motion invariant and continuous with 
respect to a suitable topology.  Among them we also provide a specific description of those which are additionally monotone.
\end{abstract}

\section{Introduction}

%If $X$ denotes a space of functions (all defined on a common domain,
%usually $\R^N$ or the unit sphere $\sfe$), 

A {\em valuation} on a space of functions $X$ is an application $\mu\,:\,X\to\R$ such that
\begin{equation}\label{valutazione}
\mu(f\vee g)+\mu(f\wedge g)=\mu(f)+\mu(g)
\end{equation}
for every $f,\,g\in X$ s.t. $f\vee g, f\wedge g\in X$; here ``$\vee$'' and ``$\wedge$'' denote the point-wise maximum and minimum, respectively. The condition
\eqref{valutazione} can be interpreted as a finite additivity property (typically verified by integrals). 
%The goal is usually to characterize all possible valuations on a given space $X$, 
%continuous with respect to a certain topology and invariant under the composition with a specific group of transformations of the domain, like
%translations, rigid motions, etc. 

The study of valuations on spaces of functions stems principally from the theory of valuations on classes of sets, in which the main current concerns 
{\em convex bodies}. We recall that a convex body is simply a compact convex subset of $\R^N$, and the family of convex bodies is 
usually denoted by $\K^N$. An application $\sigma\,:\,\K^N\to\R$ is called a valuation if 
\begin{equation}\label{valutazione 2}
\sigma(K\cup L)+\sigma(K\cap L)=\sigma(K)+\sigma(L)
\end{equation}
for every $K,L\in\K^N$ such that $K\cup L\in\K^N$ (note that the intersection of convex bodies is a convex body). Hence, in passing from 
\eqref{valutazione 2} to \eqref{valutazione} union and intersection are replaced by maximum and minimum respectively. A motivation is that 
the characteristic function of the union (resp. the intersection) of two sets is the maximum (resp. the minimum) of their characteristic
functions. 

The theory of valuations is an important branch of modern convex geometry (the theory of convex bodies). The reader is referred to 
the monograph \cite{Schneider} for an exhaustive description of the state of the art in this area, and for the corresponding 
bibliography. The valuations on $\K^N$, continuous with respect to the Hausdorff metric and rigid motion invariant, 
have been completely classified in a celebrated result by Hadwiger (see \cite{Hadwiger}, \cite{Klain}, \cite{Klain-Rota}). Hadwiger's theorem asserts that 
any valuation $\sigma$ with these properties can be written in the form
\begin{equation}\label{H intro}
\sigma(K)=\sum_{i=0}^N c_i\,V_i(K)\quad\forall\, K\in\K^N,
\end{equation}
where $c_1,\dots,c_N$ are constants and $V_1,\dots,V_N$ denote the {\em intrinsic volumes} (see section \ref{section 2}, for the definition). 
This fact will be of great importance for the results presented here. 

\medskip

Let us give a brief account of the main known results in the area of valuations on function spaces. Wright, in his PhD thesis \cite{Wright} and 
subsequently in collaboration with Baryshnikov and Ghrist \cite{BGW}, characterized rigid motion invariant and continuous valuations on 
the class of {\em definable functions} 
(we refer to the quoted papers for the definition). Their result is very similar to Hadwiger's theorem; roughly speaking it asserts that 
every valuation is the linear combination of integrals of intrinsic volumes of level sets. This type of valuations will be crucial in our results as well.

Rigid motion invariant and continuous valuations on $L^p(\R^N)$ and on $L^p(\sfe)$ ($1\le p<\infty$) 
have been studied and classified by Tsang in \cite{Tsang-2010}. 
Basically, Tsang proved that every valuation $\mu$ with these properties is of the type
\begin{equation}\label{intro Tsang}
\mu(f)=\int\phi(f)dx
\end{equation}
(here the integral is performed on $\R^N$ or $\sfe$) for some function $\phi$ defined on $\R$ verifying suitable growth conditions. Subsequently, the results
of Tsang have been extended to Orlicz spaces by Kone in \cite{Kone}. Also, the special case $p=\infty$ was studied by Cavallina in \cite{Cavallina}.

Valuations on the space of functions of bounded variations and on Sobolev spaces have been recently studied by Wang and Ma respectively, in \cite{Wang},
\cite{Wang-thesis}, \cite{Ma} and \cite{Ma-Thesis}. 

In \cite{Colesanti-Cavallina} the authors consider rigid motion invariant and continuous valuations (with respect to a certain topology that will be recalled later on)
on the space of convex functions, and found some partial characterization results under the assumption of monotonicity and homogeneity.

Note that the results that we have mentioned so far concern {\em real-valued} valuations, but there are also studies
regarding other types of valuations (e.g. matrix-valued valuations, or Minkowski and Blaschke valuations, etc.) that are interlaced with the results mentioned perviously.
A strong impulse to these studies have been given by Ludwig in the works \cite{Ludwig-2011}, \cite{Ludwig-2012}, \cite{Ludwig-2013}; the reader is referred also to
\cite{Tsang-2011} and \cite{Ober}.

\medskip

Here we consider the space $\C^N$ of {\em quasi-concave} functions of $N$ real variables. A function $f\,:\,\R^N\to\R$ is quasi-concave 
if it is non-negative and for every $t>0$ the level set
$$
L_t(f)=\{x\in\R^N\,:\,f(x)\ge t\}
$$
is (either empty or) a compact convex set. $\C^N$ includes log-concave functions and characteristic functions of convex bodies as significant examples. 

We consider valuations $\mu\,:\,\C^N\to\R$ which are rigid motion invariant, i.e.
$$
\mu(f)=\mu(f\circ T)
$$
for every $f\in\C^N$ and for every rigid motion $T$ of $\R^N$. We also impose a continuity condition on $\mu$: if $f_i$, $i\in\N$, is a {\em monotone} sequence in 
$\C^N$, converging to $f\in\C^N$ point-wise in $\R^N$, then we must have
$$
\lim_{i\to\infty}\mu(f_i)=\mu(f).
$$
In section \ref{discussion on topology} we provide some motivation for this definition, comparing this notion of continuity with other possible choices.

There is a simple way to construct valuations on $\C^N$. To start with, note that if $f,g\in\C^N$ and $t>0$
\begin{equation}\label{intro1}
L_t(f\vee g)=L_t(f)\cup L_t(g),\quad
L_t(f\wedge g)=L_t(f)\cap L_t(g).
\end{equation}
Let $\psi$ be a function defined on $(0,\infty)$ and fix $t_0>0$. Define, for every 
$f\in\C^N$, 
$$
\mu_0(f)=V_N(L_{t_{0}}(f))\psi(t_0).
$$
Using \eqref{intro1} and the additivity of volume we easily deduce that $\mu_0$ is a rigid motion invariant valuation. 
More generally, we can overlap valuations of this type at
various levels $t$, and we can further replace $V_N$ by any intrinsic volume $V_k$:
\begin{equation}\label{intro2}
\mu(f)=\int_{(0,\infty)}V_k(L_t(f))\psi(t)\,dt
=\int_{(0,\infty)}V_k(L_t(f))\,d\nu(t), \quad f\in\C^N,
\end{equation}
where $\nu$ is the measure with density $\psi$.
This is now a rather ample class of valuations; as we will see, basically every {\em monotone} valuation on $\C^N$ can be written in this form. 
To proceed, we observe that the function 
$$
t\,\rightarrow\,V_k(L_t(f))
$$
is decreasing. In particular it admits a distributional derivative which is a non-positive measure. For ease of notation
we write this measure in the form $-S_k(f;\cdot)$ where now $S_k(f;\cdot)$ is a (non-negative) Radon measure on $(0,\infty)$.
Then, integrating by parts in \eqref{intro2} (boundary terms can be neglected, as it will be clear in the sequel) we obtain:
\begin{equation}\label{intro3}
\mu(f)=\int_{(0,\infty)}\phi(t)\,dS_k(f;t)
\end{equation}
where $\phi$ is a primitive of $\psi$. Our first result is the fact that functionals of this type exhaust, by linear combinations, all possible 
rigid motion invariant and continuous valuations on $\C^N$. 

\begin{thm}\label{characterization} A map $\mu\,:\,\C^N\to\R$ is an invariant and continuous valuation on $\C^N$  if and only if there exist $(N+1)$ continuous 
functions $\phi_k$, $k=0,\dots,N$ defined on $[0,\infty)$,
\begin{equation}\label{intro0.1}
\mu(f)=\sum_{k=0}^N\int_{[0,\infty)}\phi_k(t)dS_k(f;t)\quad\forall\, f\in\C^N.
\end{equation}
Moreover, there exists $\delta>0$ such that $\phi_k\equiv0$ in $[0,\delta]$ for every $k=1,\dots,N$. 
\end{thm}
The condition that each $\phi_k$, except for $\phi_0$, vanishes in a right neighborhood of the origin guarantees that the integral
in \eqref{intro3} is finite for every $f\in\C^N$ (in fact, it is equivalent to this fact).
As in the case of Hadwiger theorem, the proof of this result is based on a preliminary step in which valuations that are additionally {\em simple} are
classified. A valuation $\mu$ on $\C^N$ is called simple if 
$$
\mbox{$f=0$ a.e. in $\R^N$}\quad
\Rightarrow\quad\mu(f)=0.
$$
Note that for $f\in\C^N$, being zero a.e. is equivalent to say that the dimension of the support of $f$ (which is a convex set) is strictly smaller than $N$.
The following result is in a sense analogous to the so-called {\em volume theorem} for convex bodies.

\begin{thm}\label{characterization simple} A map $\mu\,:\,\C^N\to\R$ is an invariant, continuous and simple valuation on $\C^N$ if and only if there exists a continuous 
function $\phi$ defined on $[0,\infty)$, with $\phi\equiv0$ in $[0,\delta]$ for some $\delta>0$, such that
\begin{equation}\label{intro0.2}
\mu(f)=\int_{\R^n} \phi(f(x))dx\quad\forall\, f\in\C^N,
\end{equation}
or, equivalently,
$$
\mu(f)=\int_{[0,\infty)}\phi(t)dS_N(f;t).
$$
\end{thm}
Here the equivalence of the two formulas follows from the layer cake principle. The representation formula of Theorem \ref{characterization} becomes more 
legible in the case of monotone valuations. Here, each term of the sum is clearly a weighted mean of the intrinsic volumes of the level sets of $f$. 

\begin{thm}\label{characterization monotone} A map $\mu$ is an invariant, continuous and monotone increasing valuation on $\C^N$ if and only if there exist
$(N+1)$ Radon measures on $[0,\infty)$, $\nu_k$, $k=0,\dots,N$, such that
\begin{equation}\label{intro0.3}
\mu(f)=\sum_{k=0}^N
\int_{[0,\infty)}V_k(L_t(f))\,d\nu_k(t),\quad\forall\, f\in\C^N.
\end{equation}
Moreover, each $\nu_k$ is non-atomic and, for $k\ge1$, there exists $\delta>0$ such that the support of $\nu_k$ is contained in $[\delta,\infty)$. 
\end{thm}

As we already mentioned, and it will be  explained in details in section \ref{connection}, the passage
$$
\int_{[0,\infty)}\phi_k(t)dS_k(f;t)\;
\longrightarrow\;
\int_{[0,\infty)}V_k(L_t(f))\,d\nu_k(t)
$$
is provided merely by an integration by parts, when this is permitted by the regularity of the function $\phi_k$. 

\medskip

The paper is organized as follows. In the next section we provide some notion from convex geometry. Section 3 is devoted to the basic properties 
quasi-convex functions, while in section 4 we define various types of valuations on the space $\C^N$. In section 5 we introduce the integral valuations,
which occur in Theorems \ref{characterization} and \ref{characterization monotone}. Theorem \ref{characterization simple} is proved in section 6, while 
sections 6 and 7 contain the proofs of Theorems \ref{characterization} and \ref{characterization monotone}, respectively.

\section{Notations and preliminaries}\label{section 2}
\noindent
We work in the $N$-dimensional Euclidean space $\mathbb{R}^{N}$, $N\ge1$, endowed with the usual scalar product $(\cdot,\cdot)$ and norm $\|\cdot\|$. 
Given a subset $A$ of $\mathbb{R}^{N}$, $\interno(A)$, $\cl(A)$ and $\partial A$ denote the interior, the closure and the topological boundary 
of $A$, respectively. For every $x\in\R^N$ and $r\ge0$, $B_r(x)$ is the closed ball of radius $r$ centered at $x$; in particular, for simplicity
we will write $B_r$ instead of $B_r(0)$.  A {\em rigid motion} of $\R^N$ will be the composition of a translation and a rotation of $\R^N$. 
The Lebesgue measure in $\mathbb{R}^{N}$ will be denoted by $V_N$.

\subsection{Convex bodies}

We recall some notions and results from convex geometry that will be used in the sequel. Our main reference on this subject is 
the monograph by Schneider \cite{Schneider}.
As stated in the introduction, the class of convex bodies is denoted by $\K^N$. For $K,L\in{\mathcal K}^{N}$, we define the \textit{Hausdorff distance} of $K$ and $L$ as
$$
\delta(K,H)=\max\{\sup_{x\in K}{\rm dist}(x,H),\sup_{y\in H}{\rm dist}(K,y)\}.
$$
Accordingly, a sequence of convex bodies $\{K_{n}\}_{n\in\mathbb{N}}\subseteq{\mathcal K}^{N}$ is 
said to converge to $K\in{\mathcal K}^{N}$ if 
$$
\delta(K_{n},K)\rightarrow 0,\ \textnormal{as}\ n\rightarrow+\infty.
$$

\begin{remark}
\textnormal{${\mathcal K}^{N}$ with respect to Hausdorff distance is a complete metric space.}
\end{remark}

\begin{remark}
\textnormal{
For every convex subset $C$ of $\R^N$, and consequently for convex bodies, its dimension ${\rm dim}(C)$ can be defined as follows:
${\rm dim}(C)$ is the smallest integer $k$ such that there exists an affine sub-space of $\R^N$ of dimension $k$, containing $C$.}
\end{remark}

\noindent
We are ready, now, to introduce some functionals operating on ${\mathcal K}^{N}$, the intrinsic volumes, which will be of fundamental importance in this paper. Among
the various ways to define intrinsic volumes, we choose the one based on the Steiner formula. Given a convex body $K$ and $\epsilon>0$, the {\em parallel
set} of $K$ is
$$
K_\epsilon=\{x\in\R^N\,|\,{\rm dist}(x,K)\le\epsilon\}.
$$  
The following result asserts that the volume of the parallel body is a polynomial in $\epsilon$, and contains the definition of intrinsic volumes. 

\begin{thm}[\bf Steiner formula]
There exist $N$ functions $V_{0},...,V_{N-1}:{\mathcal K}^{N}\rightarrow \mathbb{R_{+}}$ such that, for all $K\in {\mathcal K}^{N}$ and for all $\epsilon\ge0$, we have
$$
V_{N}(K_\epsilon)=\sum_{i=0}^{N}V_{i}(K)\omega_{N-i}\epsilon^{N-i},
$$
where $\omega_j$ denotes the volume of the unit ball in the space $\R^j$. $V_{0}(K),\dots,V_{N}(K)$ are called the intrinsic volumes of $K$.
\end{thm}

In particular, one of the intrinsic volumes is the Lebesgue measure. Moreover $V_{0}$ is the Euler characteristic, so that for every $K$ we have $V_{0}(K)=1$.
The name intrinsic volumes comes from the following fact: assume that $K$ has dimension $j\in\{0,\dots,N\}$, then $K$ can be seen as a subset of $\R^j$ 
and $V_j(K)$ is the Lebesgue measure of $K$ as a subset of $\R^j$. Intrinsic volumes have many other properties, listed in the following proposition.

\begin{proposition}[{\bf Properties of intrinsic volumes.}]\label{intrinsic volumes} 
For every $k\in\{0,\dots,N\}$ the function $V_k$ is:
\begin{itemize}
\item rigid motion invariant;
\item continuous with respect to the Hausdorff metric;
\item monotone increasing: $K\subset L$ implies $V_k(K)\le V_k(L)$; 
\item a valuation:
$$
V_k(K\cup L)+ V_k(K\cap L)=V_k(K)+V_k(L)\quad\forall\,
K,L\in\K^N\;\mbox{s.t. $K\cup L\in\K^N$.}
$$
\end{itemize}
\end{proposition}
We also set conventionally
$$
V_k(\varnothing)=0,\quad\forall\,k=0,\dots,N.
$$

The previous properties essentially characterize intrinsic volumes as stated by the following result proved by Hadwiger, already mentioned in the 
introduction.

\begin{thm}[\bf Hadwiger]\label{t2}
If $\sigma$ is a continuous and rigid motion invariant valuation, then there exist $(N+1)$ real coefficients $c_{0},...,c_{N}$ such that
$$
\sigma(K)=\sum_{i=0}^{N}c_{i}V_{i}(K),
$$
for all $K\in{\mathcal K}^{N}\cup\{\varnothing\}$.
\end{thm}

The previous theorem claims that $\{V_{0},...,V_{N}\}$ spans the vector space of all continuous and invariant valuations on ${\mathcal K}^{N}\cup\{\varnothing\}$. It can be 
also proved that $V_{0},...,V_{N}$ are linearly independent, so they form a basis of this vector space. In Hadwiger's Theorem continuity can be replaced by 
monotonicty hypothesis, obtaining the following results.

\begin{thm}\label{t3}
If $\sigma$ is a monotone increasing (resp., decreasing) rigid motion invariant valuation, then there exist $(N+1)$ coefficients $c_{0},...,c_{N}$ such that $c_{i}\geq 0$ 
(resp $c_i\le0$) for every $i$ and
$$
\sigma(K)=\sum_{i=0}^{N}c_{i}V_{i}(K),
$$
for all $K\in{\mathcal K}^{N}\cup\{\varnothing\}$.
\end{thm}

A special case of the  preceding results concerns {\em simple} valuations. A valuation $\mu$ is said to be simple if 
$$
\mu(K)=0\quad\forall\, K\in\K^N\;\mbox{s.t. $\dim(K)<N$.}
$$

\begin{corollary}[\bf Volume Theorem]\label{c1}
Let $\sigma:{\mathcal K}^{N}\cup\{\varnothing\}\rightarrow \mathbb{R}$ be a rigid motion invariant, simple and continuous valuation. Then there exists a constant 
$c$ such that
$$
\mu=cV_{N}.
$$
\end{corollary}

\begin{remark} {\rm In the previous theorem continuity can be replaced by the following weaker assumption: for every 
{\em decreasing} sequence $K_i$, $i\in\N$, in $\K^N$, converging to $K\in\K^N$, 
$$
\lim_{i\to\infty}\sigma(K_i)=\sigma(K).
$$
This follows, for instance, from the proof of the volume theorem given in \cite{Klain}.} 
\end{remark}

\section{Quasi-concave functions}\label{section 3}

\subsection{The space ${\mathcal C}^N$}

\begin{defn} A function $f:\mathbb{R}^{N}\rightarrow \mathbb{R}$ is said to be {\em quasi-concave} if
\begin{itemize}
\item $f(x)\geq 0$ for every $x\in \mathbb{R}^{N}$,
\item for every $t>0$, the set
$$
L_{t}(f)=\{x\in\mathbb{R}^{N}:\ f(x)\geq t\}
$$
is either a convex body or is empty.
\end{itemize}
We will denote with ${\mathcal C}^N$ the set of all quasi-concave functions defined on $\mathbb{R}^N$.
\end{defn}

Typical examples of quasi-convex functions are (positive multiples of) characteristic functions of convex bodies.
For $A\subseteq \mathbb{R}^{N}$ we denote by $I_{A}$ its characteristic function
$$
I_{A}:\mathbb{R}^{N}\rightarrow\mathbb{R},\quad 
I_{A}(x)=\begin{cases}
1\; &\mbox{if $x\in A$,}\\
0\; &\mbox{if $\notin A$.}
\end{cases}
$$
Then we have that $s\,I_{K}\in{\mathcal C}^N$ for every $s>0$ and $K\in{\mathcal K}^{N}$. We can also describe the sets $L_{t}(sI_{K})$, indeed
$$
L_{t}(s\,I_{K})=\begin{cases}
\varnothing\; &\mbox{if $t>s$},\\
K\; &\mbox{if $0<t\leq s$.}
\end{cases}
$$

The following proposition gathers some of the basic properties of quasi-concave functions.

\begin{proposition}\label{p1}
If $f\in{\mathcal C}^N$ then
\begin{itemize}
\item $\displaystyle{\lim_{||x||\to+\infty}} f(x)=0$,
\item $f$ is upper semi-continuous,
\item $f$ admits a maximum in $\R^n$, in particular
$$
\displaystyle{\sup_{\mathbb{R}^{N}}} f<+\infty.
$$
\end{itemize}
\end{proposition}

\begin{proof} To prove the first property, let $\epsilon>0$; as $L_\epsilon(f)$ is compact, there exists $R>0$ such that
$L_\epsilon(f)\subset B_R$. This is equivalent to say that
$$
f(x)\le\epsilon\quad\forall\, x\quad\mbox{s.t.}\quad\|x\|\ge R.
$$
Upper semi-continuity follows immediately from compactness of super-level sets. 
Let $M=\sup_{\R^N}f$ and assume that $M>0$. Let $x_n$, $n\in\N$, be a maximizing sequence:
$$
\lim_{n\to\infty} f(x_n)=M.
$$
As $f$ decays to zero at infinity, the sequence $x_n$ is compact; then we may assume that it converges to $\bar x\in\R^N$. Then,
by upper semi-continuity
$$
f(\bar x)\ge\lim_{n\to\infty}f(x_n)=M.
$$
\end{proof}

For simplicity, given $f\in\C^N$, we will denote by $M(f)$ the maximum of $f$ in $\R^N$.

\begin{remark} 
\textnormal{Let $f\in\C^N$, we denote with $\textnormal{supp}(f)$ the support of $f$, that is 
$$
\supp (f)=\cl(\{x\in\mathbb{R}^{N}:\ f(x)>0 \}).
$$
This is a convex set; indeed
$$
\supp(f)=
\bigcup_{k=1}^\infty\{x\in\R^N\,:\,f(x)\ge 1/k\}.
$$
The sets
$$
\{x\in\R^N\,:\,f(x)\ge 1/k\}\quad k\in\N,
$$
forms an increasing sequence of convex bodies and their union is convex. 
}
\end{remark}

\begin{remark}
\textnormal{A special sub-class of quasi-concave functions is that formed by {\em log-concave} functions. Let $u$ be a function defined on all $\mathbb{R}^{N}$, 
with values in $\R\cup\{+\infty\}$, convex and such that $\lim_{||x||\to+\infty} f(x)=+\infty$. Then the function $f=e^{-u}$ is quasi-concave (here we adopt the convention
$e^{-\infty}=0$). If $f$ is of this form is said to be a log-concave function.
}
\end{remark}

\subsection{Max and min of quasi-concave functions}
Let $f,g:\mathbb{R}^{N}\rightarrow \mathbb{R}$; we define the point-wise maximum and minimum  function between $f$ and $g$ as
$$
f\vee g(x)=\max\{f(x),g(x)\},\quad
f\wedge g(x)=\min\{f(x),g(x)\},
$$
for all $x\in\mathbb{R}^{N}$. These operations, applied on ${\mathcal C}^{N}$, will replace the union and intersection in the definition of valuations on 
${\mathcal K}^{N}\cup\{\varnothing\}$. The proof of the following equalities is straightforward.

\begin{lemma}\label{l1} If $f$ and $g$ belong to $\C^N$ and $t>0$:
$$
L_{t}(f\wedge g)=L_{t}(f)\cap L_{t}(g),\quad
L_{t}(f\vee g)=L_{t}(f)\cup L_{t}(g).
$$
\end{lemma}
As the intersection of two convex bodies is still a convex body, we have the following consequence.
\begin{corollary}
For all $f,g\in {\mathcal C}^N$, $f\wedge g \in {\mathcal C}^N$.
\end{corollary}
On the other hand, in general $f,g\in\C^N$ does not imply that $f\vee g$ does, as it is shown by the example in which $f$ and $g$ are characteristic functions 
of two convex bodies with empty intersection.

The following lemma follows from the definition of quasi-concave function and the fact that if $T$ is a rigid motion of $\R^N$ and 
$K\in\K^N$, then $T(K)\in\K^N$.

\begin{lemma}\label{l3}
Let $f\in {\mathcal C}^N$ be a quasi concave function and $T:\mathbb{R}^{N}\rightarrow \mathbb{R}^{N}$ a rigid motion, then $f\circ T\in {\mathcal C}^N$.
\end{lemma}

\subsection{Three technical lemmas}
We are going to prove some lemmas which will be useful for the study of continuity of valuations.

\begin{lemma}\label{lemma tecnico 1}
Let $f\in {\mathcal C}^{N}$. For all $t>0$, except for at most countably many values, we have
$$
L_{t}(f)=\cl(\{x\in\mathbb{R}^{N}:\ f(x)>t\}).
$$
\end{lemma}
\begin{proof}
We fix $t>0$ and we define
$$
\Omega_{t}(f)=\{x\in\mathbb{R}^{N}:\ f(x)>t\},\quad
H_{t}(f)=\cl(\Omega_t(f)).
$$
$\Omega_{t}(f)$ is a convex set for all $t>0$, indeed
$$
\Omega_t=\bigcup_{k\in\N}L_{t+1/k}(f).
$$
Consequently $H_{t}$ is a convex body and $H_{t}\subseteq L_{t}(f)$. We define $D_{t}=L_{t}(f)\setminus H_{t}$; our aim is now to prove
that the set of all $t>0$ such that $D_t\ne\varnothing$ is at most countable. We first note that if $K$ and $L$ are convex bodies with 
$K\subset L$ and $L\setminus K\ne\varnothing$ then $\interno(L\setminus K)\ne\varnothing$, therefore
\begin{equation}\label{dim. lemma 1}
D_t\ne\varnothing\quad\Leftrightarrow\quad V_N(D_t)>0.
\end{equation}
It follows from 
$$
D_t=L_{t}(f)\setminus H_{t}\subseteq L_{t}(f)\setminus \Omega_{t}(f)=\{x\in\mathbb{R}^{N}:\ f(x)=t\},
$$
that
\begin{equation}\label{e2}
t_{1}\neq t_{2}\ \Rightarrow\ D_{t_{1}}(f)\cap D_{t_{2}}(f)=\varnothing.
\end{equation}

For the rest of the proof we proceed by induction on $N$. For $N=1$, we observe that if $f$ is identically zero, then the lemma is trivially true. 
If $\textnormal{supp}(f)=\{x_{0}\}$ and $f(x_{0})=t_{0}>0$, then we have
$$
L_{t}(f)=\{x_0\}=\cl(\Omega_{t}(f))\quad \forall\,t>0,\ t\neq t_{0},
$$
and in particular the lemma is true. We suppose next that $\interno(\supp(f))\neq\varnothing$; let $t_{0}>0$ be a number such that 
$\textnormal{dim}(L_{t}(f))=1$, for all $t\in(0,t_{0})$ and $\textnormal{dim}(L_{t}(f))=0$, for all $t>t_{0}$.
Moreover, let $t_{1}=\max_{\mathbb{R}}f\ge t_0$. We observe that
$$
L_{t}(f)=\cl(\Omega_{t}(f))=\varnothing\quad \forall\, t>t_{1}\quad
\mbox{and}\quad
L_{t}(f)=\cl(\Omega_{t}(f))\quad \forall\, t\in(t_{0},t_{1}).
$$
Next we deal with values of $t\in(0,t_0)$. 
Let us fix $\epsilon>0$ and let $K$ be a compact set in $\R$ such that $K\supseteq L_{t}(f)$ for every $t\geq\epsilon$. We define, for $i\in\mathbb{N}$,
$$
T^{\epsilon}_{i}=\left\{t\in[\epsilon,t_{0}) :\ V_{1}(D_t)\geq\frac{1}{i}\right\}.
$$
As $D_t\subseteq K$ for all $t\ge\epsilon$ and taking \eqref{e2} into account we obtain that 
$T^{\epsilon}_{i}$ is finite . So 
$$
T^{\epsilon}=\bigcup_{i\in\mathbb{N}}T^{\epsilon}_{i}
$$
is countable for every $\epsilon>0$. By \eqref{dim. lemma 1}
$$
\{
t\ge\epsilon\,:\,D_t\ne\varnothing
\}\quad\mbox{is countable}
$$
for every $\epsilon>0$, so that 
$$
\{
t>0\,:\,D_t\ne\varnothing
\}
$$
is also countable. The proof for $N=1$ is complete.

Assume now that the claim of the lemma is true up to dimension $(N-1)$, and let us prove in dimension $N$. 
If the dimension of $\supp(f)$ is strictly smaller than $N$, then (as $\supp(f)$ is convex)
there exists an affine subspace $H$ of $\R^N$, of dimension $(N-1)$, containing $\supp(f)$. In this case the assert of the lemma follows 
applying the induction assumption to the restriction of $f$ to $H$. Next, we suppose that there exists $t_{0}>0$ such that
$$
\dim(L_{t}(f))=N,\quad \forall\, t\in(0,t_{0})
$$
and
$$
\dim(L_{t}(f))<N,\quad\forall\, t>t_{0}.
$$
By the same argument used in the one-dimensional case we can prove that 
$$
\{t\in(0,t_{0}):\ D_t\ne\varnothing\}
$$ 
is countable. For $t>t_{0}$, there exists a $(N-1)$-dimensional affine sub-space of $\R^N$ containing $L_t(f)$ for every $t>t_0$. 
To conclude the proof we apply the inductive hypothesis to the restriction of $f$ to this hyperplane.
\end{proof}

\begin{lemma}\label{l7}
Let $\{f_{i}\}_{i\in\mathbb{N}}\subseteq{\mathcal C}^{N}$ and $f\in {\mathcal C}^{N}$. Assume that $f_{i}\nearrow f$ point-wise in $\R^N$ as $i\rightarrow+\infty$. 
Then, for all $t>0$, except at most for countably many values,
$$
\lim_{i\to\infty}L_t(f_i)=L_t(f).
$$ 
\end{lemma}
\begin{proof}
For every $t>0$, the sequence of convex bodies $L_t(f_i)$, $i\in\N$, is increasing and $L_t(f_i)\subset L_t(f)$ for every $i$. 
In particular this sequence admits a limit $L_t\subset L_t(f)$. We choose $t>0$ such that
$$
L_t(f)=\cl(\{x\in\R^N\,:\,f(x)>t\}).
$$
By the previous lemma we know that this condition holds for every $t$ except at most countably many values. It is clear that for every $x$ s.t. $f(x)>t$ we have 
$x\in L_t$, hence $L_t\supset\{x\in\R^N\,:\,f(x)>t\}$; on the other hand, as $L_t$ is closed, we have that $L_t\supset L_t(f)$. Hence $L_t=L_t(f)$ and the 
proof is complete.
\end{proof}

\begin{lemma}\label{lemma tecnico 3}
Let $\{f_{i}\}_{i\in\mathbb{N}}\subseteq{\mathcal C}^{N}$ and $f\in {\mathcal C}^{N}$. Assume that $f_{i}\searrow f$ point-wise in $\R^N$ as $i\rightarrow+\infty$. 
Then for all $t>0$
$$
\lim_{i\to\infty}L_t(f_i)=L_t(f).
$$ 
\end{lemma}
\begin{proof}
The sequence $L_t(f_i)$ is decreasing and its limit, denoted by $L_t$, contains $L_t(f)$. On the other hand, as now
$$
L_t=\bigcap_{k\in\N} L_t(f_k)
$$
(see Lemma 1.8.1 of \cite{Schneider}),
if $x\in L_t$ then $f_i(x)\ge t$ for every $i$, so that $f(x)\ge t$ i.e. $x\in L_t(f)$. 
\end{proof}

\section{Valuations on $\C^N$}\label{section 4}

\begin{defn}
A functional $\mu:{\mathcal C}^{N}\rightarrow \mathbb{R}$ is said to be a valuation if
\begin{itemize}
\item $\mu(\underline{0})=0$, where $\underline{0}\in{\mathcal C}^{N}$ is the function identically equal to zero;
\item for all $f$ and $g\in {\mathcal C}^{N}$ such that $f\vee g\in{\mathcal C}^{N}$, we have
$$
\mu(f)+\mu(g)=\mu(f\vee g)+\mu(f\wedge g).
$$
\end{itemize}
\end{defn}

A valuation $\mu$ is said to be rigid motion invariant, or simply invariant, if for every rigid motion $T:\mathbb{R}^{N}\rightarrow \mathbb{R}^{N}$ 
and for every $f\in{\mathcal C}^{N}$, we have
$$
\mu(f)=\mu(f\circ T).
$$
In this paper we will always consider invariant valuations. We will also need a notion of continuity which is expressed by the following definition.

\begin{defn}
A valuation $\mu$ is said to be continuous if for every sequence $\{f_{i}\}_{i\in\mathbb{N}}\subseteq {\mathcal C}^{N}$ and $f\in{\mathcal C}^{N}$ such that 
$f_i$ converges point-wise to $f$ in $\R^N$, and $f_i$ is either monotone increasing or decreasing w.r.t. $i$, we have
$$
\mu(f_{i})\rightarrow\mu(f),\ for\ i\rightarrow +\infty.
$$
\end{defn}

To conclude the list of properties that a valuation may have and that are relevant to our scope, we say that a valuation $\mu$ is monotone increasing
(resp. decreasing) if, given $f,g\in\C^N$, 
$$
\mbox{$f\le g$ point-wise in $\R^N$ implies $\mu(f)\le\mu(g)$ (resp. $\mu(f)\ge\mu(g)$).}
$$

\subsection{A brief discussion on the choice of the topology in $\C^N$}\label{discussion on topology}

A natural choice of a topology in $\C^N$ would be the one induced by point-wise convergence. Let us see that this choice would too restrictive,
with respect to the theory of continuous and rigid motion invariant (but translations would be enough) valuations. Indeed, any translation invariant valuation
$\mu$ on $\C^N$ such that
$$
\lim_{i\to\infty}\mu(f_i)=\mu(f)
$$
for every sequence $f_i$, $i\in\N$, in $\C^N$, converging to some $f\in\C^N$ point-wise, must be the valuation constantly equal to 0. To prove this claim,
let $f\in\C^N$ have compact support, let $e_1$ be the first vector of the canonical basis of $\R^N$ and set
$$
f_i(x)=f(x-i\,e_1)\quad\forall\, x\in\R^N,\quad\forall\, i\in\N.
$$
The sequence $f_i$ converges point-wise to the function $f_0\equiv0$ in $\R^N$, so that, by translation invariance, and as $\mu(f_0)=0$,
we have $\mu(f)=0$. Hence $\mu$ vanishes on each function $f$ with compact support. On the other hand every element of $\C^N$ is the point-wise limit
of a sequence of functions in $\C^N$ with compact support. Hence $\mu\equiv0$. 

\medskip

A different choice could be based on the following consideration: we have seen that $\C^N\subset L^\infty(\R^N)$, hence it inherits the topology of this space.
In \cite{Cavallina}, Cavallina studied translation invariant and continuous valuations on $L^\infty(\R^N)$. In particular he proved that there exists non-trivial
translation invariant and continuous valuations on this space, which vanishes on functions with compact support. In particular they can not be written 
in integral form as those found in the present paper. Nothing that in dimension $N=1$ translation and rigid motion invariance provide basically the same
condition, this suggests that the choice of the topology on $L^\infty(\R^N)$ on $\C^N$ would lead us to a completely different type of valuations.

\section{Integral valuations}\label{section 5}
A class of examples of invariant valuations, which will be crucial for our characterization results, is that of integral valuations.

\subsection{Continuous integral valuations}\label{sec. Continuous and integral valuations}

Let $k\in\{0,\dots,N\}$. For $f\in\C^N$, consider the function
$$
t\,\rightarrow\,u(t)=V_k(L_t(f))\,\quad
t>0.
$$
This is a decreasing function, which vanishes for $t>M(f)=\max_{\R^N}f$. In particular $u$ has bounded variation in
$[\delta,M(f)]$ for every $\delta>0$, hence there exists a Radon measure defined in $(0,\infty)$, that we will denote by 
$S_k(f;\cdot)$, such that 
$$
\mbox{$-S_k(f;\cdot)$ is the distributional derivative of $u$}
$$
(see, for instance, \cite{AFP}). Note that, as $u$ is decreasing, we have put a minus sign in this definition to have a non-negative measure.
The support of $S_k(f;\cdot)$ is contained in $[0,M(f)]$. 

Let $\phi$ be a continuous function defined on $[0,\infty)$, such that $\phi(0)=0$. We consider the functional on $\C^N$ defined by
\begin{equation}\label{integral valuation 1}
\mu(f)=\int_{(0,\infty)}\phi(t) dS_k(f;t)\quad f\in\C^N.
\end{equation}
The aim of this section is to prove that this is a continuous and invariant valuation on $\C^N$. As a first step,
we need to find some condition on the function $\phi$ which guarantee that the above integral is well defined for every $f$. 

Assume that
\begin{equation}\label{int cond 2}
\mbox{$\exists\,\delta>0$ s.t. $\phi(t)=0$ for every $t\in[0,\delta]$.}
\end{equation}
Then 
\begin{eqnarray*}
\int_{(0,\infty)}\phi_+(t) dS_k(f;t)&=&
\int_{[\delta,M(f)]}\phi_+(t)\,dS_k(f,t)\\
&\le& M\,\left(
V_k(L_\delta(f))-V_k(M(f))\right)<\infty,
\end{eqnarray*}
where $M(f)=\max_{\R^N}f$, $M=\max_{[\delta,\max_{\R^N}f]}\phi_+$ and $\phi_+$ is the positive part of $\phi$. Analogously we can prove that
the integral of the negative part of $\phi$, denoted by $\phi_-$, is finite, so that $\mu$ is well defined. 

We will prove that, for $k\ge1$, condition \eqref{int cond 2} is necessary as well. Clearly, if $\mu(f)$ is well defined 
(i.e. is a real number) for every $f\in\C^N$, then
$$
\int_{(0,\infty)}\phi_+(t) dS_k(f;t)<\infty\quad
\mbox{and}\quad
\int_{(0,\infty)}\phi_-(t) dS_k(f;t)<\infty\quad
\forall\,f\in\C^N.
$$
Assume that $\phi_+$ does not vanish identically in any right neighborhood of the origin. Then we have
$$
\psi(t):=\int_0^t \phi_+(\tau)\,d\tau>0\quad\forall\,t>0.
$$
The function 
$$
t\,\rightarrow\,h(t)=\int_t^1\frac1{\psi(s)}ds,\quad t\in(0,1],
$$
is strictly decreasing. As $k\ge 1$, we can construct a function $f\in\C^N$ such that
\begin{equation}\label{relazione inversa}
V_k(L_t(f))=h(t)\quad\mbox{for every $t>0$.}
\end{equation}
Indeed, consider a function of the form
$$
f(x)=w(\|x\|),\quad x\in\R^N,
$$
where $w\in C^1([0,+\infty))$ is positive and strictly decreasing. Then $f\in C^N$ and
$L_t(f)=B_{r(t)}$, where 
$$
r(t)=w^{-1}(t)
$$
for every $t\in(0,f(0)]$ (note that $f(0)=M(f)$). Hence
$$
V_k(L_t(f))=c\,(w^{-1}(t))^k
$$
where $c$ is a positive constant depending on $k$ and $N$. Hence if we choose
$$
w=\left[
\left(
\frac1c\,h
\right)^{1/k}
\right]^{-1},
$$
\eqref{relazione inversa} is verified. Hence
$$
dS_k(f;t)=\frac1{\psi(t)}dt,
$$
and
$$
\int_{(0,\infty)}\phi_+(t)dS_k(f;t)=
\int_{(0,M(f))}\frac{\psi'(t)}{\psi(t)}dt=\infty.
$$
In the same way we can prove that $\phi_-$ must vanish in a right neighborhood of the origin. We have proved the following result.

\begin{lemma} Let $\phi\in C([0,\infty))$ and $k\in\{1,\dots,N\}$. Then $\phi$ has finite integral with respect to the measure
$S_k(f;\cdot)$ for every $f\in\C^N$ if and only if $\phi$ verifies \eqref{int cond 2}.
\end{lemma}

In the special case $k=0$, as the intrinsic volume $V_0$ is the Euler characteristic,
$$
u(t)=
\left\{
\begin{array}{lll}
1\quad&\mbox{if $0<t\le M(f)$,}\\
0\quad&\mbox{if $t>M(f)$.}
\end{array}
\right.
$$
That is, $S_0$ is the Dirac point mass measure concentrated at $M(f)$ and $\mu$ can be written as
$$
\mu(f)=\phi(M(f))\quad\forall, f\in\C^N.
$$ 

Next we show that \eqref{integral valuation 1} defines a continuous and invariant valuation. 

\begin{proposition}\label{proposition integral valuations 1} 
Let $k\in\{0,\dots,N\}$ and $\phi\in C([0,\infty))$ be such that $\phi(0)=0$. If $k\ge1$ assume that 
\eqref{int cond 2} is verified. Then \eqref{integral valuation 1} defines an invariant and continuous valuation on $\C^N$. 
\end{proposition}

\begin{proof} For every $f\in\C^N$ we define the function $u_f\,:\,[0,M(f)]\to\R$ as
$$
u_f(t)=V_k(L_t(f)).
$$
As already remarked, this is a decreasing function. In particular it has bounded variation in $[\delta,M(f)]$.
Let $\phi_i$, $i\in\N$, be a sequence of functions in $C^\infty([0,\infty))$, with compact support, converging uniformly to $\phi$ on compact sets. 
As $\phi\equiv0$ in $[0,\delta]$, we may assume that the same holds for every $\phi_i$. 
Then we have
$$
\mu(f)=\lim_{i\to\infty}\mu_i(f),
$$
where
$$
\mu_i(f)=\int_{[0,\infty)}\phi_i(t)dS_k(f;t)\quad\forall\, f\in\C^N.
$$
By the definition of distributional derivative of a monotone function we have, for every $f$ and for every $i$:
$$
\int_{[0,\infty)}\phi_i(t)dS_k(f;t)=
\int_{[0,\infty)}u_f(t)\phi_i'(t)dt=
\int_{[0,M(f)]}V_k(L_t(f))\phi_i'(t)dt.
$$
On the other hand, if $f,g\in\C^N$ are such that $f\vee g\in\C^N$, for every $t>0$
\begin{equation}\label{inclusion-exclusion for functions}
L_t(f\vee g)=L_t(f)\cup L_t(g),\quad
L_t(f\wedge g)=L_t(f)\cap L_t(g).
\end{equation}
As intrinsic volumes are valuations
$$
V_k(L_t(f\vee g))+V_k(L_t(f\wedge g))=
V_k(L_t(f))+V_k(L_t(g)).
$$
Multiplying both sides times $\phi'_i(t)$ and integrating on $[0,\infty)$ we obtain
$$
\mu_i(f\vee g)+\mu_i(f\wedge g)=\mu_i(f)+\mu_i(g).
$$
Letting $i\to\infty$ we deduce the valuation property for $\mu$. 

In order to prove the continuity of $\mu$, we first consider the case $k\ge 1$. 
Let $f_i, f\in\C^N$, $i\in\N$, and assume that the sequence
$f_i$ is either increasing or decreasing with respect to $i$, and it converges point-wise to $f$ in $\R^N$.
Note that in each case there exists a constant $M>0$ such that $M(f_i), M(f)\le M$ for every $i$. 
Consider now the sequence of functions $u_{f_i}$. By the monotonicity of the sequence $f_i$, and that of 
intrinsic volumes, this is a monotone sequence of decreasing functions, and it converges a.e. to $u_f$ in 
$(0,\infty)$, by Lemmas \ref{l7} and \ref{lemma tecnico 3}. In particular the sequence $u_{f_i}$ has uniformly bounded
total variation in $[\delta, M]$. Consequently, the sequence of measures $S_k(f_i;\cdot)$, $i\in\N$, converges weakly to the measure $S_k(f;\cdot)$
as $i\to\infty$. Hence, as $\phi$ is continuous
$$
\lim_{i\to\infty}\mu(f_i)=\lim_{i\to\infty}\int_{[\delta, M]}\phi(t)\,dS_k(f_i;t)
=\int_{[0,M]}\phi(t)\,dS_k(f;t)=\mu(f).
$$

If $k=0$ then we have seen that
$$
\mu(f)=\phi(M(f))\quad\forall\,f\in\C^N.
$$
Hence in this case continuity follows from the following fact: if $f_i$, $i\in\N$, is a monotone sequence in $\C^N$ converging point-wise to 
$f$, then 
$$
\lim_{i\to\infty}M(f_i)=M(f).
$$
This is a simple exercise that we leave to the reader.

Finally, the invariance of $\mu$ follows directly from the invariance of intrinsic volumes with respect to rigid motions.
\end{proof}

\subsection{Monotone (and continuous) integral valuations}

In this section we introduce a slightly different type of integral valuations, which will be needed to characterize all possible continuous and monotone 
valuations on $\C^N$. Note that, as it will be clear in the sequel, when the involved functions are smooth enough, the two types 
(i.e. of the present and of previous section) can be reduced one to another by an integration by parts. 

Let $k\in\{0,\dots,N\}$ and let $\nu$ be a Radon measure on $\left(0,+\infty\right)$; assume that  
\begin{equation}\label{e1}
\int_{0}^{+\infty}V_{k}(L_{t}(f))d\nu(t)<+\infty,\quad \forall f\in{\mathcal C}^{N}.
\end{equation}
We will return later on explicit condition on $\nu$ such that \eqref{e1} holds. Then define the functional $\mu\,:\,\C^N\to\R$ by
\begin{equation}\label{integral valuations}
\mu(f)=\int_{0}^{+\infty}V_{k}(L_{t}(f))d\nu(t)\quad\forall\, f\in\C^N.
\end{equation}

\begin{proposition} Let $\nu$ be a Radon measure on $(0,\infty)$ which verifies \eqref{e1}; then the functional defined by \eqref{integral valuations} is a rigid motion 
invariant and monotone increasing valuation.
\end{proposition}

\begin{proof} The proof that $\mu$ is a valuation follows from \eqref{inclusion-exclusion for functions} and the valuation property for intrinsic volumes, as in the 
proof of Proposition \ref{proposition integral valuations 1}. The same can be done for invariance. as for monotonicity, note that if $f,g\in\C^N$ and $f\le g$,
then 
$$
L_t(f)\subset L_t(g)\quad\forall\, t>0.
$$
Therefore, as intrinsic volumes are monotone, $V_k(L_t(f))\le V_K(L_t(g))$ for every $t>0$.
\end{proof}

If we do not impose any further assumption the valuation $\mu$ needs not to be continuous. Indeed, for example, if we fix $t=t_{0}>0$ and let 
$\nu=\delta_{t_{0}}$ be the delta Dirac measure at $t_0$; then the valuation 
$$
\mu(f)=V_{N}(L_{t_{0}}(f)),\ \forall f\in{\mathcal C}^{N},
$$ 
is not continuous. To see it, let $f=t_0 I_{B_1}$ (recall that $B_1$ is the unit ball of $\R^N$) and let
$$
f_i=t_{0}\left(1-\frac1i\right)I_{B_1}\quad\forall\,i\in\N.
$$
Then $f_i$ is a monotone sequence of elements of $\C^N$ converging point-wise to $f$ in $\R^N$. On the other hand
$$
\mu(f_i)=0\quad\forall\,i\in\N,
$$
while $\mu(f)=V_N(B_1)>0$. The next results asserts that the presence of atoms is the only possible cause of discontinuity for $\mu$. We recall that
a measure $\nu$ defined on $[0,\infty)$ is said non-atomic if $\nu(\{t\})=0$ for every $t\ge0$. 

\begin{proposition}\label{p3}
Let $\nu$ be a Radon measure on $\left(0,+\infty\right) $ such that $(\ref{e1})$ holds and let $\mu$ be the valuation defined by $(\ref{e1})$. Then the two following conditions are equivalent:
\begin{itemize}
\item[i)] $\nu$ is non-atomic,
\item[ii)] $\mu$ is continuous.
\end{itemize}
\end{proposition}
\begin{proof}
Suppose that {\em i)} does not hold, than there exists $t_{0}$ such that $\nu(\{t_{0}\})=\alpha>0$. Define  
$\varphi:\mathbb{R_{+}}\rightarrow\mathbb{R}$ by
$$
\varphi(t)=\int_{(0,t]}d\nu(s).
$$
$\varphi$ is an increasing function with a jump discontinuity at $t_0$ of amplitude $\alpha$.
Now let $f=t_{0}I_{B_1}$ and $f_i=t_{0}(1-\frac1i)I_{B_1}$, for $i\in\N$. Then $f_i$ is an increasing sequence in $\C^N$, converging point-wise
to $f$ in $\R^N$. On the other hand
$$
\mu(f)=\int_{0}^{t_{0}}V_{k}(B)d\nu(s)=V_{k}(B)\nu((0,t_{0}])=V_k(B_1)\,\varphi(t_0)
$$
and similarly
$$
\mu(f_{i})=V_k(B_1)\,\varphi\left(t_0-\frac 1i\right).
$$
Consequently 
$$
\lim_{i\rightarrow+\infty}\mu(f_{i})<\mu(f).
$$ 

Vice versa, suppose that {\em i)} holds. We observe that, as $\nu$ is non-atomic, every countable subset 
has measure zero with respect to $\nu$. Let $f_{i}\in{\mathcal C}^{N}$, $i\in\N$, be a sequence 
such that either $f_{i}\nearrow f$ or $f_{i}\searrow f$ as $i\rightarrow+\infty$,
point-wise in $\mathbb{R}^{N}$, for some $f\in{\mathcal C}^{N}$. Set 
$$
u_i(t)=V_k(L_t(f_i)),\quad
u(t)=V_k(L_t(f))\quad\forall\, t\ge0,\quad\forall\, k\in\N.
$$
The sequence $u_i$ is monotone and, by Lemmas \ref{l7} and \ref{lemma tecnico 3}, converges to $u$ $\nu$-a.e.
Hence, by the continuity of intrinsic volumes and the monotone convergence theorem, we obtain
$$
\lim_{i\to\infty}\mu(f_i)=
\lim_{i\to\infty}\int_{(0,\infty)}u_i(t)\,d\nu=
\int_{(0,\infty)}u(t)\,d\nu(t)=\mu(f).
$$
\end{proof}

Now we are going to find a more explicit form of condition \eqref{e1}. We need the following lemma.

\begin{lemma}\label{l4}
Let $\phi:\left[ 0,+\infty\right)\rightarrow\mathbb{R}$ be an increasing, non negative and continuous function with $\phi(0)=0$ 
and $\phi(t)>0$, for all $t>0$. Let $\nu$ be a Radon measure such that $\phi(t)=\nu([0,t])$, for all $t\geq 0$.
Then 
$$
\int_{0}^{1}\frac{1}{\phi^{k}(t)}d\nu(t)=+\infty,\ \forall k\geq 1.
$$
\end{lemma}

\begin{proof} Fix $\alpha\in[0,1]$. The function $\psi\,:\,[\alpha,1]\to\R$ defined by
$$
\psi(t)=
\left\{
\begin{array}{lll}
\dfrac1{k-1}\phi^{1-k}(t)\quad&\mbox{if $k>1$,}\\
\\
\ln(\phi(t))\quad&\mbox{if $k=1$,}
\end{array}
\right.
$$
is continuous and with bounded variation in $[\alpha,1]$. Its distributional derivative is
$$
\frac1{\phi^k(t)}\,\nu.
$$
Hence, for $k>1$,
$$
\frac1{k-1}[\phi^{1-k}(\alpha)-\phi^{1-k}(1)]=
\psi(1)-\psi(\alpha)=
\int_{[\alpha,1]}\frac{d\nu}{\phi^k(t)}.
$$ 
The claim of the lemma follows letting $\alpha\to0^+$. A similar argument can be applied to the case $k=1$.
\end{proof}

\begin{proposition}\label{p4}
Let $\nu$ be a non-atomic Radon measure on $\left[ 0,+\infty\right) $ and let $k\in\{1,\dots,N\}$. 
Then $(\ref{e1})$ holds if and only if:
\begin{equation}\label{integrability condition}
\mbox{$\exists$ $\delta>0$ such that $\nu([0,\delta])=0$.}
\end{equation}
\end{proposition}

\begin{proof}
We suppose that there exists $\delta>0$ such that $[0,\delta]\cap \textnormal{supp}(\nu)=\varnothing$. Then we have,
for every $f\in\C^N$, 
\begin{eqnarray}
\mu(f)&=&
\int_{\delta}^{M(f)}V_{i}(L_{t}(f))d\nu(t)\leq V_{i}(L_{\delta}(f))\int_{\delta}^{M(f)}d\nu(t)\\
&=&V_{i}(L_{\delta}(f))(\nu([0,M(f)])-\nu([0,\delta]))<+\infty.
\end{eqnarray}
with $M(f)=\max_{\mathbb{R}^{N}}f$.

\noindent
Vice versa, assume that $(\ref{e1})$ holds. By contradiction, we suppose that for all $\delta>0$, we have $\ \nu([0,\delta])>0$.
We define 
$$
\phi(t)=\nu([0,t]),\quad t\in[0,1]
$$  
then $\phi$ is continuous (as $\nu$ is non-atomic) and increasing; moreover $\phi(0)=0$ and $\phi(t)>0$, for all $t>0$.
The function 
$$
\psi(t)=\frac1{t\phi(t)},\quad t\in(0,1],
$$
is continuous and strictly decreasing. Its inverse $\psi^{-1}$ is defined in $[\psi(1),\infty)$; we extend it to $[0,\psi(1))$ setting
$$
\psi^{-1}(r)=1\quad\forall\, r\in[0,\psi(1)).
$$
Then
$$
V_{1}(\{r\in\left[ 0,+\infty\right):\ \psi^{-1}(r)\geq t\})=
\left\{
\begin{array}{lll}
\psi(t),\quad&\forall\, t\in\left( 0,1\right]\\
\\
0\quad&\forall\,t>1.
\end{array}
\right.
$$
We define now the function $f\,:\,\R^N\to\R$ as 
$$
f(x)=\psi^{-1}(||x||),\quad\forall\, x\in\R^N.
$$
Then 
$$
L_{t}(f)=\{x\in\mathbb{R}^{N}:\ \psi(||x||)\geq t\}=B_{\frac{1}{t\phi(t)}}(0),
$$
and
$$
V_{k}(L_{t}(f))=c\,\frac{1}{t^k\phi^{k}(t)} \quad\forall\,t\in(0,1],
$$
where $c>0$ depends on $N$ and $k$. Hence, by Lemma \ref{l4}
$$
\int_{0}^{+\infty}V_{k}(L_{t}(f))d\nu(t)=
\int_{0}^1V_{k}(L_{t}(f))d\nu(t)
\ge c\,\int_{0}^{+\infty}\frac{d\nu(t)}{\phi^{k}(t)}=+\infty.
$$
\end{proof}

The following proposition summarizes some of the results we have found so far.

\begin{proposition}\label{integral valuations 1} Let $k\in\{0,\dots,N\}$ and let $\nu$ be a Radon measure on $[0,\infty)$ which is non atomic and,
if $k\ge1$, verifies condition
\eqref{integrability condition}. Then the map $\mu\,:\,\C^N\to\R$ defined by \eqref{integral valuations} is an invariant, continuous and increasing
valuations.
\end{proposition}

\subsection{The connection between the two types of integral valuations}\label{connection}
When the regularity of the involved functions permits, the two types of integral valuations that we have seen 
can be obtained one from each other by a simple integration by parts (up to decomposing an arbitrary aluation as the difference of two 
monotone valuations).

Let $k\in\{0,\dots,N\}$ and $\phi\in C^1([0,\infty))$ be such that $\phi(0)=0$. For simplicity,
we may assume also that $\phi$ has compact support. Let $f\in\C^N$. By the definition of distributional derivative of an increasing function we have:
$$
\int_{[0,\infty)}\phi(t)\,dS_k(f;t)=\int_{[0,\infty)}\phi'(t) V_k(L_t(f))dt.
$$
If we further decompose $-\phi'$ as the difference of two non-negative functions, and we denote by $\nu_1$ and $\nu_2$ the Radon measures
having those functions as densities, we get
$$
\int_{[0,\infty)}\phi(t)\,dS_k(f;t)=\int_{[0,\infty)}V_k(L_t(f))d\nu_1(t)-
\int_{[0,\infty)}V_k(L_t(f))d\nu_2(t).
$$
The assumption that $\phi$ has compact support can be demoved by a standard approximation argument. In his way
we have seen that each valuation of the form \eqref{integral valuation 1}, if $\phi$ is regular, is the difference of two 
monotone integral valuations of type \eqref{integral valuations}.

Vice versa, let $\nu$ be a Radon measure (with support contained in $[\delta,\infty)$, for some $\delta>0$), 
and assume that it has a smooth density with respect to the Lebesgue measure: 
$$
d\nu(t)=\phi'(t)dt
$$
where $\phi\in C^1([0,\infty))$, and it has compact support. Then 
$$
\int_{[0,\infty)}V_k(L_t(f))\,d\nu(t)=
\int_{[0,\infty)}\phi(t)\, dS_k(f;t).
$$
Also in this case the assumption that the support of $\nu$ is compact can be removed. 
In other words each integral monotone valuation, with sufficiently smooth density, can be written in the form \eqref{integral valuation 1}.

\subsection{The case $k=N$}

If $\mu$ is a valuation of the form \eqref{integral valuation 1} and $k=N$, the layer cake principle provides an alternative 
simple representation.

\begin{proposition}\label{label}
Let $\phi$ be a continuous function on $[0,\infty)$ verifying \eqref{integrability condition}. Then for every $f\in\C^N$ we have
\begin{equation}\label{prima}
\int_{[0,\infty)}\phi(t)\,dS_N(f;t)=
\int_{\R^N}\phi(f(x))dx.
\end{equation}
\end{proposition}

\begin{proof} As $\phi$ can be written as the difference of two non-negative continuous function, and \eqref{prima} is linear with respect to $\phi$, 
there is no restriction if we assume that $\phi\ge0$. In addition we suppose initially that $\phi\in C^1([0,\infty))$ and it has compact support. 
Fix $f\in\C^N$; by the definition of distributional derivative, we have
$$
\int_{[0,\infty)}\phi(t)\,dS_N(f;t)=\int_{[0,\infty)}V_N(L_t(f))\phi'(t)dt.
$$
There exists $\phi_1,\phi_2\in C^1([0,\infty))$, strictly increasing, such that $\phi=\phi_1-\phi_2$. Now:
$$
\int_{[0,\infty)}V_N(L_t(f))\phi_1'(t)dt=
\int_{[0,\infty)}V_N(\{x\in\R^N\,:\,\phi_1(f(x))\ge s\})ds=
\int_{\R^N}\phi_1(f(x))dx,
$$
where in the last equality we have used the layer cake principle. Applying the same argument to $\phi_2$ we obtain \eqref{prima}
when $\phi$ is smooth and compactly supported. 
For the general case, we apply the result obtained in the previous part of the proof
to a sequence $\phi_i$, $i\in\N$, of functions in $C^1([0,\infty))$, with compact support, which converges uniformly to $\phi$ on compact subsets of
$(0,\infty)$. The conclusion follows from a direct application of the dominated convergence theorem.
\end{proof}

\section{Simple valuations}\label{section 6}

Throughout this section $\mu$ will be an invariant and continuous valuation on $\C^N$. We will also assume that $\mu$
is {\em simple}. 

\begin{defn}
A valuation $\mu$ on ${\mathcal C}^{N}$ is said to be simple if, for every $f\in\C^N$ with 
$\dim(\supp(f))<N$, we have $\mu(f)=0$.
\end{defn}
Note that $\dim(\supp(f))<N$ implies that $f=0$ a.e. in $\R^N$, hence each valuation of the form \eqref{prima} is simple. We are going to prove that
in fact the converse of this statement is true.

Fix $t\ge0$ and define a real-valued function $\sigma_t$ on ${\mathcal K}^{N}\cup\{\varnothing\}$ as 
$$
\sigma_{t}(K)=\mu(tI_{K})\quad\forall\, K\in\K^N,\quad\sigma_t(\varnothing)=0.
$$ 
Let $K,L\in\K^N$ be such that $K\cup L\in\K^N$. As, trivially,
$$
tI_K\vee tI_L=tI_{K\cup L}\quad\mbox{and}\quad
tI_K\wedge tI_L=tI_{K\cap L},
$$
using the valuation property of $\mu$ we infer
$$
\sigma_t(K\cup L)+\sigma_t(K\cap L)=\sigma_t(K)+\sigma_t(L),
$$
i.e. $\sigma_t$ is a valuation on $\K^N$. It also inherits directly two properties of $\mu$: it is invariant and simple. Then, by the continuity of 
$\mu$, Corollary \ref{c1} and the subsequent remark, there exists a constant $c$ such that
\begin{equation}\label{polytopes}
\sigma_t(K)=cV_N(K)
\end{equation}
for every $K\in\K^N$.  The constant $c$ will in general depend on $t$, i.e. it is a real-valued function defined in $[0,\infty)$. We denote this function by 
$\phi_N$. Note that, as $\mu(f)=0$ for $f\equiv0$, $\phi_N(0)=0$. Moreover, the continuity of $\mu$ implies that for every $t_0\ge0$ and for every monotone 
sequence $t_i$, $i\in\N$, converging to $t_0$, we have
$$
\phi_N(t_0)=\lim_{i\to\infty}\phi_N(t_i).
$$
From this it follows that $\phi_N$ is continuous in $[0,\infty)$. 

\begin{proposition}\label{fi N} 
Let $\mu$ be an invariant, continuous and simple valuation on $\C^N$. Then there exists a continuous function $\phi_N$ on
$[0,\infty)$, such that 
$$
\mu(tI_K)=\phi_N(t)\,V_N(K)
$$
for every $t\ge0$ and for every $K\in\K^N$.
\end{proposition}

\subsection{Simple functions}

\begin{defn} A function $f\,:\,\R^N\to\R$ is called simple if it can be written in the form
\begin{equation}\label{simple function}
f=t_1 I_{K_1}\vee\dots\vee t_m I_{K_m}
\end{equation}
where $0<t_1<\dots<t_m$ and $K_1,\dots,K_m$ are convex bodies such that
$$
K_1\supset K_2\supset\dots\supset K_m.
$$
\end{defn}

The proof of the following fact is straightforward.

\begin{proposition}\label{level sets of simple functions}
Let $f$ be a simple function of the form \eqref{simple function} and let $t>0$. 
Then 
\begin{equation}
L_t(f)=\{x\in\R^N\,:\,f(x)\ge t\}=
\left\{
\begin{array}{lllll}
K_i\quad&\mbox{if $t\in(t_{i-1},t_i]$ for some $i=1,\dots m$,}\\
\\
\varnothing\quad&\mbox{if $t>t_m$,}
\end{array}
\right.
\end{equation}
where we have set $t_0=0$.
\end{proposition}

In particular simple functions are quasi-concave. Let $k\in\{0,\dots,N\}$, and let $f$ be of the form \eqref{simple function}. Consider the function
$$
t\,\rightarrow\,u(t):=V_k(L_t(f)),\quad t>0.
$$
By Proposition \ref{level sets of simple functions}, this is a decreasing function that is constant on each interval of the form
$(t_{i-1},t_i]$, on which it has the value $V_k(K_i)$. Hence its distributional derivative is $-S_k(f;\cdot)$, where
\begin{equation}\label{measure for simple functions}
S_k(f;\cdot)=\sum_{i=1}^{m-1}(V_k(K_{i})-V_k(K_{i+1}))\,\delta_{t_i}(\cdot)+V_k(K_m)\delta_{t_m}(\cdot).
\end{equation}

\subsection{Characterization of simple valuations}

In this section we are going to prove Theorem \ref{characterization simple}. Note that one implication, i.e. that every map of the form \eqref{intro0.2} 
has the required properties, follows from  the results of the previous section; in particular Proposition \ref{proposition integral valuations 1}
and Proposition \ref{label}

We will first prove it for simple functions and then pass to the general case by approximation.

\begin{lemma}\label{simple on simple} Let $\mu$ be an invariant, continuous and simple valuation on $\C^N$, and let 
$\phi=\phi_N$ be the function whose existence is established in Proposition \ref{fi N}. Then, for every simple function $f\in\C^N$ we have
$$
\mu(f)=\int_{[0,\infty)}\phi(t)\,dS_N(f;t).
$$
\end{lemma}

\begin{proof} Let $f$ be of the form \eqref{simple function}. We prove the following formula
\begin{equation}\label{induction}
\mu(f)=\sum_{i=1}^{m-1}\phi(t_i)(V_N(K_i)-V_N(K_{i+1}))+\phi(t_m)V_N(K_m);
\end{equation}
by \eqref{measure for simple functions}, this is equivalent to the statement of the lemma. Equality \eqref{induction} will be proved by induction on $m$.
For $m=1$ its validity follows from Proposition \ref{fi N}. Assume that it has been proved up to $(m-1)$. Set
$$
g=t_1 I_{K_1}\vee\dots\vee t_{m-1}I_{K_{m-1}},\quad
h=t_m I_{K_m}.
$$
We have that $g,h\in\C^N$ and 
$$
g\vee h=f\in\C^N,\quad 
g\wedge h= t_{m-1} I_{K_m}.
$$
Using the valuation property of $\mu$ and Proposition \ref{fi N} we get
\begin{eqnarray*}
\mu(f)&=&\mu(g\vee h)=\mu(g)+\mu(h)-\mu(g\wedge h)\\
&=&\mu(g)+\phi(t_m)V_N(K_m)-\phi(t_{m-1})V_N(K_m).
\end{eqnarray*}
On the other hand, by induction
\begin{equation*}
\mu(g)=\sum_{i=1}^{m-2}\phi(t_i)(V_N(K_i)-V_N(K_{i+1}))+\phi(t_{m-1})V_N(K_{m-1}).
\end{equation*}
The last two equalities complete the proof.
\end{proof}

\noindent
{\em Proof of Theorem \ref{characterization simple}.}  
As before, $\phi=\phi_N$ is the function coming from Proposition \ref{fi N}. We want to prove that
\begin{equation}\label{claim theorem 1}
\mu(f)=\int_{[0,\infty)}\phi(t)\,dS_N(f;t)
\end{equation}
for every $f\in\C^N$.

\medskip

\noindent
{\bf Step 1.}
Our first step is to establish the validity of this formula when the support of $f$ bounded, i.e. there exists
some convex body $K$ such that
\begin{equation}\label{compact support}
L_t(f)\subset K\quad\forall\, t>0.
\end{equation}
Given $f\in\C^N$ with this property, we build a monotone sequence of simple functions, $f_i$, $i\in\N$, converging point-wise to $f$ in $\R^N$.  
Let $M=M(f)$ be the maximum of $f$ on $\R^N$. Fix $i\in\N$. We consider the dyadic partition ${\mathcal P}_i$ of $[0,M]$:
$$
{\mathcal P}_i=\left\{t_j=j\,\frac M{2^i}\,:\,j=0,\dots, 2^i\right\}.
$$
Set 
$$
K_j=L_{t_j}(f),\quad
f_i=\bigvee_{j=1}^{2^i}t_jI_{K_j}.
$$
$f_i$ is a simple function; as $t_j I_{K_j}\le f$ for every $j$ we have that $f_i\le f$ in $\R^N$. 
The sequence of function $f_i$ is increasing, since ${\mathcal P}_i\subset{\mathcal P}_{i+1}$. 
The inequality
$f_i\le f$ implies that
$$
\lim_{i\to\infty} f_i(x)\le f(x)\quad\forall\, x\in\R^N
$$
(in particular the support of $f_i$ is contained in $K$, for every $i\in\N$).
We want to establish the reverse inequality. Let $x\in\R^N$; if $f(x)=0$ then trivially
$$
f_i(x)=0\quad\forall\,i\quad\mbox{hence}\quad
\lim_{i\to\infty} f_i(x)=f(x).
$$
Assume that $f(x)>0$ and fix $\epsilon>0$. Let $i_0\in\N$ be such that $2^{-i_0}M<\epsilon$. Let $j\in\{1,\dots,2^{i_0}-1\}$ be such that
$$
f(x)\in\left(
j\,\frac M{2^{i_0}},(j+1)\frac M{2^{i_0}}
\right].
$$
Then 
$$
f(x)\le j\,\frac M{2^{i_0}}+\frac M{2^{i_0}}\le
f_{i_0}(x)+\epsilon
\le\lim_{i\to\infty}f_i(x)+\epsilon.
$$
Hence the sequence $f_i$ converges point-wise to $f$ in $\R^N$. In particular, by the continuity of $\mu$ we have that
$$ 
\mu(f)=\lim_{i\to\infty}\mu(f_i)=
\lim_{i\to\infty}\int_{[0,\infty)}\phi(t)\,dS_N(f_i;t).
$$
By Lemma \ref{l7}, a further consequence is that
$$
\lim_{i\to\infty}u_i(t)=u(t)\quad\mbox{for a.e. $t\in(0,\infty)$,}
$$
where
$$
u_i(t)=V_N(L_t(f_i)),\quad i\in\N,\quad
u(t)=V_N(L_t(f))
$$
for $t>0$. We consider now the sequence of measures $S_N(f_i;\cdot)$, $i\in\N$; the total variation of these measures 
in $(0,\infty)$ is uniformly bounded by $V_N(K)$, moreover they are all supported in $(0,M)$. As they are the distributional derivatives of the 
functions $u_i$, which converges a.e. to $u$, we have that (see for instance \cite[Proposition 3.13]{AFP}) the sequence 
$S_N(f_i;\cdot)$ converges weakly in the sense of measures to $S_N(f;\cdot)$. 
This implies that
\begin{equation}\label{proof of theorem 1 - 1}
\lim_{i\to\infty}\int_{(0,\infty)}\bar\phi(t)\,dS_N(f_i;t)=\int_{(0,\infty)}\bar\phi(t)\,dS_N(f;t)
\end{equation}
for every function $\bar\phi$ continuous in $(0,\infty)$, such that $\bar\phi(0)=0$ and $\bar\phi(t)$ is identically zero
for $t$ sufficiently large. In particular (recalling that $\phi(0)=0$), we can take $\bar\phi$ such that it equals $\phi$ in 
$[0,M]$. Hence, as the support of the measures $S_N(f_i;\cdot)$ is contained in this interval, we have that
\eqref{proof of theorem 1 - 1} holds for $\phi$ as well. This proves the validity of \eqref{claim theorem 1} for functions with bounded support.

\medskip

\noindent
{\bf Step 2.} This is the most technical part of the proof. The main scope here is to prove that $\phi$ is identically zero in some
right neighborhood of the origin. Let $f\in\C^N$. For $i\in\N$, let 
$$
f_i=f\wedge(M(f)I_{B_i})
$$
where $B_i$ is the closed ball centered at the origin, with radius $i$. The function $f_i$ coincides with $f$ in $B_i$ and vanishes in $\R^N\setminus B_i$;
in particular it has bounded support. Moreover, the sequence $f_i$, $i\in\N$, is increasing and converges point-wise to $f$ in $\R^N$. Hence
$$
\mu(f)=\lim_{i\to\infty}\mu(f_i)=\lim_{i\to\infty}\int_{(0,\infty)}\phi(t)\,dS_N(f_i;t).
$$
Let $\phi_+$ and $\phi_-$ be the positive and negative parts of $\phi$, respectively. We have that
$$
\lim_{i\to\infty}\left[\int_{(0,\infty)}\phi_+(t)\,dS_N(f_i;t)+\int_{(0,\infty)}\phi_-(t)\,dS_N(f_i;t)\right]
$$
exists and it is finite. We want to prove that this implies that $\phi_+$ and $\phi_-$ vanishes identically in $[0,\delta]$ for some 
$\delta>0$. 

By contradiction, assume that this is not true for $\phi_+$. Then there exists three sequences $t_i$, $r_i$ and $\epsilon_i$,
$i\in\N$, with the following properties: $t_i$ tends decreasing to zero; $r_i>0$ is such that the intervals $C_i=[t_i-r_i,t_i+r_i]$ are contained in
$(0,1]$ and pairwise disjoint; $\phi_+(t)\ge\epsilon_i>0$ for $t\in C_i$.  Let
$$
C=\bigcup_{i\in\N} C_i\,,\quad
\Omega=(0,1]\setminus C.
$$
Next we define a function $\gamma\,:\,(0,1]\to[0,\infty)$ as follows. $\gamma(t)=0$ for every $t\in\Omega$ while, for every $i\in\N$, $\gamma$ 
is continuous in $C_i$ and
$$
\gamma(t_i\pm r_i)=0,\quad
\int_{C_i}\gamma(t)dt=\frac1\epsilon_i.
$$
Note in particular that $\gamma$ vanishes on the support of $\phi_-$ intersected with $(0,1]$. We also set
$$
g(t)=\gamma(t)+1\quad\forall\, t>0.
$$
Observe that
$$
\int_0^1 \phi_-(t)g(t)dt=\int_0^1\phi_-(t)dt<\infty.
$$
On the other hand
\begin{eqnarray*}
\int_0^1\phi_+(t)g(t)dt&\ge&
\int_0^1\phi(t)\gamma(t)dt=
\sum_{i=1}^\infty\int_{C_i}\phi_+(t)\gamma(t)dt\\
&\ge&\sum_{i=1}^\infty\epsilon_i\int_{C_i}\gamma(t)dt=+\infty.
\end{eqnarray*}

Let
$$
G(t)=\int_t^1 g(s)ds\quad\mbox{and}\quad
\rho(t)=[G(t)]^{1/N},\quad 0<t\le 1.
$$
As $\gamma$ is non-negative, $g$ is strictly positive, and continuous in $(0,1)$. Hence $G$ is 
strictly decreasing and continuous, and the same holds for $\rho$. Let 
$$
S=\sup_{(0,1]}\rho=\lim_{t\to 0^+}\rho(t),
$$ 
and let $\rho^{-1}\,:\,[0,S)\to\R$ be the inverse function of $\rho$. If $S<\infty$, we extend $\rho^{-1}$ to be zero in
$[S,\infty)$. In this way, $\rho^{-1}$ is continuous in $[0,\infty)$, and $C^1([0,S))$. Let
$$
f(x)=\rho^{-1}(\|x\|),\quad\forall\, x\in\R^N.
$$
For $t>0$ we have
$$
L_t(f)=
\left\{
\begin{array}{lll}
\{x\in\R^N\,:\,\|x\|\le\rho(t)\}\quad&\mbox{if $t\le1$,}\\
\varnothing\quad&\mbox{if $t>1$.}
\end{array}
\right.
$$
In particular $f\in\C^N$. Consequently,
$$
V_N(L_t(f))=c\,\rho^N(t)=c\,G(t)\quad\forall\,t\in(0,1],
$$
where $c>0$ is a dimensional constant, and then
$$
dS_N(f;t)=c\,g(t)dt.
$$
By the previous considerations
\begin{eqnarray*}
\int_{[0,\infty)}\phi_+(t)dS_N(f,t)=c \int_{[0,\infty)}\phi_+(t)g(t)dt=\infty,\quad
\int_{[0,\infty)}\phi_+(t)dS_N(f,t)<\infty.
\end{eqnarray*}
Clearly we also have that
$$
\int_{[0,\infty)}\phi_+(t)dS_N(f,t)=
\lim_{i\to\infty}\int_{[0,\infty)}\phi_+(t)dS_N(f_i,t),
$$
and the same holds for $\phi_-$; here $f_i$ is the sequence approximating $f$ defined before. We reached a contradiction.

\medskip

\noindent
{\bf Step 3.} The conclusion of the proof proceeds as follows. Let $\bar\mu\,:\,\C^N\to\R$ be defined by
$$
\bar\mu(f)=\int_{(0,\infty)}\phi(t)\,dS_N(f;t).
$$
By the previous step, and by the results of section \ref{sec. Continuous and integral valuations}, this is well defined, and is an invariant and 
continuous valuation. Hence the same properties are shared by $\mu-\bar\mu$; on the other hand, by Step 1 and the definition of $\bar\mu$, 
this vanishes on functions with bounded support. As for any element $f$ of $\C^N$ there is a monotone sequence of functions in $C^N$, with bounded 
support and converging point-wise to $f$ in $\R^N$, and as $\mu-\bar\mu$ is continuous, it must be identically zero on $\C^N$.
\begin{flushright}
$\square$
\end{flushright}

\section{Proof of Theorem \ref{characterization}}\label{section 7}

As for the proof of Theorem \ref{characterization simple}, note that one implication of Theorem \ref{characterization} is already proved, by an application of
Proposition \ref{proposition integral valuations 1} (and its extension to the case $k=0$).

For the other implication 
we proceed by induction on $N$. For the first step of induction, let $\mu$ be an invariant and continuous valuation on $\C^1$. For $t>0$
let 
$$
\phi_0(t)=\mu(t I_{\{0\}}).
$$
This is a continuous function in $\R$, with $\phi_0(0)=0$. We consider the application $\mu_0\,:\,\C^1\to\R$:
$$
\mu_0(f)=\phi_0(M(f))
$$
where as usual $M(f)=\max_{\R}f$. By what we have seen in section \ref{sec. Continuous and integral valuations}, this is 
an invariant and continuous valuation. Note that it can be written in the form
$$
\mu_0(f)=\int_{(0,\infty)}\phi_0(t)\,dS_0(f;t).
$$
Next we set $\bar\mu=\mu-\mu_0$; this is still an invariant and continuous valuation, and it is also simple. Indeed, if $f\in\C^1$ is such that
$\dim(\supp(f))=0$, this is equivalent to say that
$$
f=tI_{\{x_0\}}
$$
for some $t\ge0$ and $x_0\in\R$.  Hence
$$
\mu(f)=\mu(tI_{\{0\}})=\phi_0(t)=\mu_0(f).
$$
Therefore we may apply Theorem \ref{characterization simple} to $\mu_1$ and deduce that there exists a function $\phi_1\in C([0,\infty))$,
which vanishes identically in $[0,\delta]$ for some $\delta>0$, and such that
$$
\bar\mu(f)=\int_{(0,\infty)}\phi_1(t)\,dS_1(f;t)\quad\forall\, f\in\C^1.
$$
The proof in the one-dimensional case is complete.

\medskip

We suppose that the Theorem holds up to dimension $(N-1)$. Let $H$ be an hyperplane of $\mathbb{R}^N$ and define 
$\C^N_H=\{f\in{\mathcal C}^N:\ \textnormal{supp}(f)\subseteq H\}$. $\C^N_H$ can be identified with ${\mathcal C}^{N-1}$; 
moreover $\mu$ restricted to $\C^N_H$ is trivially still an invariant and continuous valuation. By the induction assumption, 
there exists $\phi_k\in C([0,\infty))$, $k=0,\dots,N-1$, such that
$$
\mu(f)=\sum_{k=0}^{N-1}\int_{(0,\infty)}\phi_k(t)\,dS_{k}(f;t)\quad\forall\, f\in\C^N_H.
$$
In addition, there exists $\delta>0$ such that $\phi_1,\dots,\phi_{N-1}$ vanish in $[0,\delta]$. Let 
$\bar{\mu}:{\mathcal C}^N\rightarrow \mathbb{R}$ as
$$
\bar{\mu}(f)=\sum_{k=0}^{N-1}\int_{(0,\infty)}
\phi_k(t)\,dS_k(f;t).
$$
This is well defined for $f\in\C^N$ and it is an invariant and continuous valuation. The difference $\mu-\bar\mu$ is simple;
applying Theorem \ref{characterization simple} to it, as in the one-dimensional case, we complete the proof.
\begin{flushright}
$\square$
\end{flushright}

\section{Monotone valuations}\label{section 8}

In this section we will prove Theorem \ref{characterization monotone}. By Proposition \ref{integral valuations 1}, every map of the form 
\eqref{intro0.3} has the required properties. 

To prove the opposite implication, we will assume that $\mu$ is an invariant, continuous and 
increasing valuation on $\C^N$ throughout. Note that, as $\mu(f_0)=0$, where $f_0$ is the function identically zero in $\R^N$, we have that
$\mu(f)\ge0$ for every $f\in\C^N$.

The proof is divided into three parts.

\subsection{Identification of the measures $\nu_k$, $k=0,\dots,N$.} 

We proceed as in the proof of Proposition \ref{fi N}. 
Fix $t>0$ and consider the application $\sigma_t\,:\,\K^N\to\R$:
$$
\sigma_t(K)=\mu(tI_K),\quad K\in\K^N.
$$
This is a rigid motion invariant valuation on $\K^N$ and, as $\mu$ is increasing, $\sigma_t$ has the same property. Hence there exist 
$(N+1)$ coefficients, depending on $t$, that we denote by $\psi_k(t)$, $k=0,\dots,N$, such that
\begin{equation}\label{somma monotone}
\sigma_t(K)=\sum_{k=0}^N\psi_k(t)V_k(K)\quad\forall\, K\in\K^N.
\end{equation}
We prove that each $\psi_k$ is continuous and monotone in $(0,\infty)$. Let us fix the index $k\in\{0,\dots,N\}$, and let $\Delta_k$
be a closed $k$-dimensional ball in $\R^N$, of radius 1. We have
$$
V_j(\Delta_k)=0\quad\forall\, j=k+1,\dots,N,
$$
and
$$
V_{k}(\Delta_k)=:c(k)>0.
$$
Fix $r\ge0$; for every $j$, $V_j$ is positively homogeneous of order $j$, hence, 
for $t>0$,
$$
\mu(tI_{r\Delta_k})=\sum_{j=0}^k r^{j}V_j(\Delta_k)\psi_j(t).
$$
Consequently
$$
\psi_k(t)=V_k(\Delta_k)\cdot\lim_{r\to \infty}\frac{\mu(tI_{r\Delta_k})}{r^{k}}.
$$
By the properties of $\mu$, the function $t\to\mu(tI_{r\Delta_k})$ is non-negative, increasing and vanishes for $t=0$, for every $r\ge0$; these properties
are inherited by $\psi_k$.

\medskip

As for continuity, we proceed in a similar way. To prove that $\psi_0$ is continuous we observe that the function
$$
t\,\rightarrow\,\mu(t\Delta_0)=\psi_0(t)
$$
is continuous, by the continuity of $\mu$. Assume that we have proved that $\psi_0,\dots,\psi_{k-1}$ are continuous. Then by the equality
$$
\mu(tI_{\Delta_k})=\sum_{j=1}^k V_j(\Delta_k)\psi_j(t),
$$ 
it follows that $\psi_k$ is continuous.

\begin{proposition} Let $\mu$ be an invariant, continuous and increasing valuation on $\C^N$. Then there exists $(N+1)$ functions 
$\psi_0,\dots,\psi_N$ defined in $[0,\infty)$, such that \eqref{somma monotone} holds for every $t\ge0$ and for every $K$. In particular
each $\psi_k$ is continuous, increasing, and vanishes at $t=0$. 
\end{proposition}

For every $k\in\{0,\dots,N\}$ we denote by $\nu_k$ the distributional derivative of $\psi_k$. In particular
as $\psi_k$ is continuous, $\nu_k$ is non-atomic and
$$
\psi_k(t)=\nu_k([0,t)),\quad\forall\, t\ge0.
$$

\subsection{The case of simple functions}

Let $f$ be a simple function:
$$
f=t_1 I_{K_1}\vee\dots\vee t_mI_{K_m}
$$
with $0<t_1<\dots<t_m$, $K_1\supset\dots\supset K_m$ and $K_i\in\K^N$ for every $i$. The following formula can be proved with the same 
method used for \eqref{induction}
\begin{equation}\label{induction2}
\mu(f)=
\sum_{k=0}^N
\sum_{i=1}^m
(\psi_k(t_i)-\psi_k(t_{i-1}))V_k(L_{t_i}(f)),
\end{equation}
where we have set $t_0=0$. As
$$
\psi_k(t_i)-\psi_k(t_{i-1})=\nu_k((t_{i-1},t_i])
$$
and $L_t(f)=K_i$ for every $t\in(t_{i-1}, t_i]$, we have
\begin{equation}\label{simple functions}
\mu(f)=
\sum_{k=0}^N
\int_{[0,\infty)}V_k(L_t(f))\,d\nu_k(t).
\end{equation}
In other words, we have proved the theorem for simple functions.

\subsection{Proof of Theorem \ref{characterization monotone}}

Let $f\in\C^n$ and let $f_i$, $i\in\N$, be the sequence of functions built in the proof of Theorem \ref{characterization simple}, Step 2. We have seen that
$f_i$ is increasing and converges point-wise to $f$ in $\R^N$. In particular, for every $k=0,\dots,N$, the sequence of functions 
$V_k(L_t(f_i))$, $t\ge0$, $i\in\N$, is monotone increasing and it converges a.e. to $V_k(L_t(f))$ in $[0,\infty)$. By the B. Levi theorem, we have that
$$
\lim_{i\to\infty}\int_{[0,\infty)}V_k(L_t(f_i))\,d\nu_k(t)=
\int_{[0,\infty)}V_k(L_t(f))\,d\nu_k(t)
$$
for every $k$. Using \eqref{simple functions} and the continuity of $\mu$ we have that the representation formula \eqref{simple functions} can be extended
to every $f\in\C^N$. 

Note that in \eqref{simple function} each term of the sum in the right hand-side is non-negative, hence we have that
$$
\int_{[0,\infty)}V_k(L_t(f))\,d\nu_k(t)<\infty\quad\forall\,f\in\C^N.
$$
Applying Proposition \ref{p4} we obtain that, if $k\ge1$, there exists $\delta>0$ such that the support of $\nu_k$ is contained in
$[\delta,\infty)$. The proof is complete.
\begin{flushright}
$\square$
\end{flushright}

\end{document}